\documentclass[12pt]{amsart}

\usepackage{latexsym,amsmath,amssymb,amsthm,amsfonts}

\usepackage{pstricks}
\usepackage{graphicx}

\usepackage[shortalphabetic]{amsrefs}

\usepackage{mathrsfs}

\newcommand{\ds}{\displaystyle}
\newcommand{\R}{{\mathbb{R}}}

\newcommand{\A}{{\mathscr{A}}}
\newcommand{\B}{{\mathscr{B}}}
\newcommand{\UU}{{\mathscr{U}}}
\newcommand{\LL}{{\mathscr{L}}}

\renewcommand{\H}{{\mathscr{H}}}

\newcommand{\CA}{{\mathcal{A}}}
\newcommand{\CB}{{\mathcal{B}}}
\newcommand{\C}{{\mathcal{C}}}
\newcommand{\M}{{\mathcal{M}}}

\renewcommand{\d}{{\rm diam}}
\renewcommand{\r}{{\rm rad}}

\newtheorem{theorem}{Theorem}[section]
\newtheorem{lemma}[theorem]{Lemma}

\theoremstyle{remark}

\newtheorem{definition}[theorem]{Definition}

\renewcommand{\phi}{\varphi}

\numberwithin{equation}{section}

\title{On Banach-Mazur distance between planar convex bodies}

\author{Serhii Brodiuk}
\address{Faculty of Computer Science and Cybernetics, Taras Shevchenko National University of Kyiv, Hlushkova Avenue 4d, Kyiv, 03680, Ukraine}
\email{serhii.brodiuk@gmail.com}

\author{Nazarii Palko}
\address{Faculty of Computer Science and Cybernetics, Taras Shevchenko National University of Kyiv, Hlushkova Avenue 4d, Kyiv, 03680, Ukraine}
\email{nazarii.palko@gmail.com}

\author{Andriy Prymak}
\address{Department of Mathematics, University of Manitoba, Winnipeg, MB, R3T2N2, Canada}
\email{prymak@gmail.com}
\thanks{The third author was supported by NSERC of Canada Discovery Grant RGPIN 04863-15.}

\keywords{Planar convex body, Banach-Mazur distance, Approximation, Affine trasform, Regular hexagon}

\subjclass[2010]{52A10, 52A27, 52A40}

\begin{document}

\begin{abstract}
	Upper estimates of the diameter and the radius of the family of planar convex bodies with respect to the Banach-Mazur distance are obtained. Namely, it is shown that the diameter does not exceed $\tfrac{19-\sqrt{73}}4\approx 2.614$, which improves the previously known bound of $3$, and that the radius does not exceed $\frac{117}{70}\approx 1.671$.
\end{abstract}

\maketitle

\section{Introduction}

Let $\C^n$ be the family of all convex bodies (convex compact sets with non-empty interior) in Euclidean space $\R^n$, and $\M^n$ be the subfamily of all centrally symmetric convex bodies (representing unit balls in $n$-dimensional real Banach spaces). For $\A,\B\in\C^n$, the Banach-Mazur distance between $\A$ and $\B$ is
\[
{d}(\A,\B)=\inf_{T,h_\lambda}\{\lambda: T(\A)\subset \B \subset h_\lambda(T(\A))\},
\]
where $T$ is an affine transform and $h_\lambda$ is a homothety with ratio $\lambda>0$. %The Banach-Mazur distance is affine invariant and multiplicative. 
For two families $\CA,\CB\subset \C^n$, we extend the notation by setting 
$
{d}(\CA,\CB):=\sup\{{d}(\A,\B): \A\in \CA,\, \B\in \CB \},
$
and also let ${d}(\A,\CB):={d}(\{\A\},\CB)$ for $\A\in \C^n$. For $\CA\subset\C^n$, the Banach-Mazur diameter and radius of $\CA$ are defined as $\d(\CA):=d(\CA,\CA)$ and $\r(\CA):=\inf \{d(\A,\CA): \A\in \CA\}$, respectively.

Perhaps one of the most well-known results implying estimates on the Banach-Mazur distance is the theorem by John~\cite{Jo} characterizing the ellipsoid of largest volume inscribed into a body from $\M^n$ or from $\C^n$. As a consequence, if $\B^n$ is the unit ball in $\R^n$, then ${d}(\B^n, \M^n)=\sqrt{n}$ and $d(\B^n, \C^n)=n$, where examples of bodies with largest distance are a cube and a simplex, respectively. Hence, $\d(\M^n)\le n$, which is asymptotically sharp due to the bound $\d(\M^n)\ge C n$ established by Gluskin~\cite{Gl} (here and below $C$ denotes positive absolute constants). For the non-symmetric case, John's theorem implies $\d(\C^n)\le n^2$, but this can be improved significantly as Rudelson~\cite{Ru} established $\d(\C^n)\le C n^{4/3}(\ln(n+1))^9$. It is still an open question to find the asymptotic behavior of $\d(\C^n)$. Regarding the radii of $\M^n$ and of $\C^n$, John's theorem implies $\r(\M^n)\le \sqrt{n}$ and $\r(\C^n)\le n$. While $\r(\M^n)\ge C\sqrt{n}$ due to Gluskin's result, we do not seem to know much about lower bounds on $\r(\C^n)$ except for $\r(\C^n)\ge\sqrt{n}$ trivially obtained by considering a ball and a simplex. A generalization of John's theorem obtained in~\cite{Go} implies $d(\M^n,\C^n)=n$, so one can take any centrally symmetric body as a ``center'' to show that $\r(\C^n)\le n$. 

For the planar case $n=2$, the Banach-Mazur distance between a square and a regular hexagon is $\frac32$. It was shown by Stromquist~\cite{St} that $\r(\M^2)=\sqrt{\frac32}$ and a ``central'' body was explicitly constructed, so, consequently, $\d(\M^2)=\frac32$. Similarly to the asymptotic situation, the non-symmetric planar case appears to be a more challenging question. The bound $\d(\C^2)\le 4$ implied by the John's theorem was improved to $\d(\C^2)\le 3$ by Lassak in~\cite{La}. For the lower bound in the general case, the Banach-Mazur distance between a regular pentagon and a triangle is $1+\frac{\sqrt{5}}{2}\approx 2.118$, see~\cite{Fl} and~\cite{La-triangle}, so, in summary, $1+\frac{\sqrt{5}}{2}\le \d(\C^2)\le 3$.  It is believed that $\d(\C^2)=1+\frac{\sqrt{5}}{2}$. We improve the upper bound by proving the following:
\begin{theorem}\label{thm:diam}
	$\d(\C^2)\le \tfrac{19-\sqrt{73}}4<2.614$.
\end{theorem}

% We also mention that for $\Q$, the family of all convex quadrilaterals from $\C^2$, it was shown in~\cite{La-quadrangles} that $d(\Q,\Q)=2$.

The main geometric argument used in~\cite{La} to show that $\d(\C^2)\le 3$ is due to Besicovitch~\cite{Be}. Namely, it asserts that any $\A\in\C^2$ has an inscribed affine-regular hexagon, in other words, there exists an affine transform $T$ such that the boundary of $T(\A)$ contains the vertices of a regular hexagon $\H_1$. By convexity, this implies $\H_1\subset T(\A)\subset \H_2$, where $\H_2$ is certain regular hexagon which depends only on $\H_1$. Our key auxiliary result is an improvement of the inclusions $\H_1\subset T(\A)\subset \H_2$. %, however, we cannot achieve this with fixed bounding polygons and have to consider specific one-parametric families of polygons defined in Definition~\ref{def:polygons}. %We postpone stating our key auxiliary result (Lemma~\ref{thm:main}) till Section~\ref{sec:main} as in combination with Definition~\ref{def:polygons} it is somewhat too technical to be included in the introduction. Theorem~\ref{thm:diam} is a consequence of Lemma~\ref{thm:main}.

Now let us turn our attention to the estimates of Banach-Mazur radius of planar convex bodies. Summarizing already mentioned results, it is known that $1.455\approx \sqrt{1+\frac{\sqrt{5}}{2}}\le \r(\C^2)\le 2$. %As another application of Lemma~\ref{thm:main}, 
We obtain the following improvement of the upper bound:
\begin{theorem}\label{thm:rad}
$\r(\C^2)\le d(\A,\C^2)\le \frac{117}{70}<1.672$, where $\A$ is the $7$-gon with the vertices $(0,\frac25)$, $(\pm1,1)$, $(\pm2,2)$ and $(\pm1,3)$.
\end{theorem}
%We remark that the same method of the proof can be used to obtain a very tiny improvement of the bound $\frac{117}{70}$ (which would be bigger than $\frac{5}{3}\approx 1.667$), but that would be unnecessarily lengthy and $\frac25$ in the definition of $\A$ would need to be replaced by a more complicated expression. 

We state and prove our key auxiliary result Lemma~\ref{thm:main} in Section~\ref{sec:main}, which also includes some important technical computations. The upper bounds on the Banach-Mazur diameter and radius of the planar convex bodies are proved in Sections~\ref{sec:diam} and~\ref{sec:rad}, respectively.

\section{An auxiliary result and some computations}\label{sec:main}

\begin{definition}\label{def:polygons}
	For any $a\in[0,1]$, we define $\LL_a$ to be the convex hull of $(0,\frac{2a}{1+a})$, $(\pm1,1)$, $(\pm2,2)$ and $(\pm1,3)$. If $a\in[0,\frac12)$, we define $\UU_a$ to be the convex hull of $(\pm a,a)$, $(\pm(2+a),2-a)$, $(\pm(3-a),3-a)$, $(\pm(3-2a),3)$ and $(\pm a,4-a)$. If $a\in [\frac12,1]$, we let $\UU_a$ be the convex hull of $(\pm a,a)$, $(\pm (3-2a),1)$, $(\pm(3-a),1+a)$, $(\pm(3-a),3-a)$, $(\pm(3-2a),3)$ and $(\pm a,4-a)$. 
\end{definition}
%Note that $\LL_a$ is a $7$-gon and $\UU_a$ is either a $10$-gon or a $12$-gon except for some boundary values of $a$.

Our key auxiliary result is the following lemma.
\begin{lemma}\label{thm:main}
	For any $\A\in\C^2$ there exist $a\in[0,1]$ and an affine transform $T$ such that $\LL_a\subset T(\A)\subset \UU_a$.
\end{lemma}

Let us illustrate some intuition behind the statement of the lemma. The resulting $a$ in a certain sense measures how far is $\A$ from an inscribed affine regular hexagon. In particular, if $a=1$, then we have $T(\A)=\LL_1=\UU_1$, i.e. $\A$ is an affine regular hexagon. At the other extreme, if $a=0$, then there is a point of $\A$ (namely, $T^{-1}(0,0)$) which is at its ``furthest'' from the inscribed affine regular hexagon. In this case we have that pretty much ``half'' of the body is determined since
\[
T(\A)\cap\{(x,y):y\le 2\}=
\LL_0\cap\{(x,y):y\le 2\}=
\UU_0\cap\{(x,y):y\le 2\}=
\{(x,y):|x|\le y\le 2 \}.
\]

\begin{proof}[Proof of Lemma~\ref{thm:main}.]
	Using~\cite{Be}, we let $H_i\in\partial \A$, $i=1,\dots,6$, be the vertices of an inscribed affine regular hexagon with the center $O$, where $\partial \A$ denotes the boundary of $\A$. Let $M_i$ be the midpoint of the segment $H_iH_{i+1}$ (indices are considered modulo $6$), and let $V_i$ be the point of intersection of the lines $H_{i-1}H_i$ and $H_{i+1}H_{i+2}$ (alternatively, $\overrightarrow{OV_i}=2\overrightarrow{OM_i}$). Let $B_i$ be the point of intersection of the ray $OM_i$ with $\partial \A$, $U_i$ be the point of intersection of the lines $H_{i+1}B_i$ and $H_iV_i$, and $W_i$ be the point of intersection of the lines $H_iB_i$ and $H_{i+1}V_i$, see Figure~\ref{fig:def}. We define $a_i:=|U_iV_i|/|H_iV_i|$ (by symmetry, $a_i=|W_iV_i|/|H_{i+1}V_i|$), where $|XY|$ stands for the Euclidean distance between $X$ and $Y$, $a:=\ds\min_i a_i$, and $j$ be such that $a_j=a$. Finally, define $T$ to be the affine transform mapping $H_{j+1}$, $\dots$, $H_{j+6}$ to $(1,1)$, $(2,2)$, $(1,3)$, $(-1,3)$, $(-2,2)$, $(-1,1)$, respectively. It is straightforward to check that $T(B_j)=(0,\frac{2a}{1+a})$. 
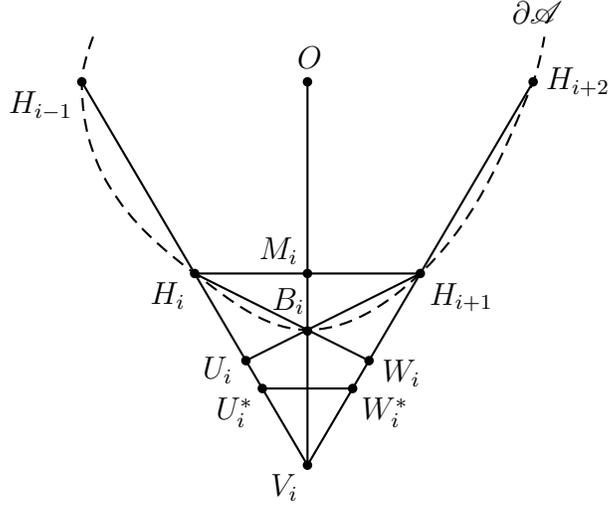
\begin{figure}[h]
	\psset{unit=1.5cm}
	\begin{pspicture}(-2.5,-1.8)(2.5,2.5)
	\psline{}(-2,1.7)(0,-1.7)
	\psline{}(-1,0)(1,0)
	\psline{}(0,-1.7)(2,1.7)
	\psline{}(0,-1.7)(0,1.7)
	\psline{}(-0.5455,-0.773)(1,0)
	\psline{}(0.5455,-0.773)(-1,0)
	\pscurve[linestyle=dashed]{}(-1.9,2.1)(-2,1.7)(-1.8,0.85)(-1,0)(0,-0.5)(1,0)(2,1.7)(2.1,2.1) 
	\uput[100](2.1,2.1){$\partial \A$}
	\psdot(-2,1.7) \uput[225](-2,1.7){$H_{i-1}$}
	\psdot(0,-0.51) \uput{0.15}[120](0,-0.51){$B_i$}
	\psdot(0,1.7) \uput[85](0,1.7){$O$}
	\psdot(0,0) \uput[135](0,0){$M_i$}
	\psdot(-1,0) \uput[225](-1,0){$H_i$}
	\psdot(1,0) \uput[-45](1,0){$H_{i+1}$}
	\psdot(-0.5455,-0.773) \uput[200](-0.5455,-0.773){$U_i$}
	\psdot(0.5455,-0.773) \uput[-25](0.5455,-0.773){$W_i$}
	\psdot(0,-1.7) \uput[225](0,-1.7){$V_i$}
	\psdot(2,1.7) \uput[0](2,1.7){$H_{i+2}$}
	\psline{}(-0.4,-1.02)(0.4,-1.02)
	\psdot(-0.4,-1.02) \uput[225](-0.4,-1.02){$U_i^*$}
	\psdot(0.4,-1.02) \uput[-45](0.4,-1.02){$W_i^*$}
	\end{pspicture}
	\caption{Location of the points $U_i$, $W_i$, $U_i^*$ and $W_i^*$}
	\label{fig:def}
\end{figure}	
	
	Since all seven points defining $\LL_a$ belong to $\partial T(\A)$, we get by convexity that $\LL_a\subset T(A)$.
	
	For every $i$, let $U_i^*$ and $W_i^*$ be the points on the segments $H_iV_i$ and $H_{i+1}V_i$ respectively such that $a=|U_i^*V_i|/|H_iV_i|=|W_i^*V_i|/|H_{i+1}V_i|$. Due to the choice of $a$, we have $|U_iV_i|\ge|U_i^*V_i|$ and $|W_iV_i|\ge|W_i^*V_i|$, so by convexity the segment $U_i^*W_i^*$ does not contain interior points of $\A$ (this segment can have common points with $\partial \A$ only for the ``degenerate'' cases when $a=0$ or $a=1$). Therefore, $\A$ is a subset of the $12$-gon with the vertices $U_i^*$ and $W_i^*$. One can easily see that $T(U_i^*)$ and $T(W_i^*)$ are exactly the points defining $\UU_a$ when $a\in[\frac12,1]$. Therefore, $T(\A)\subset \UU_a$ if $a\in[\frac12,1]$. For the remainder of the proof suppose $a\in[0,\frac12)$ . We need certain additional considerations in the triangles $H_{j+1}V_{j+1}H_{j+2}$ and $H_{j+5}V_{j+5}H_j$. Due to symmetry, let us consider only $H_{j+1}V_{j+1}H_{j+2}$. Let $C$ be the point of intersection of the lines $B_jH_{j+1}$ and $H_{j+2}V_{j+1}$.
	Since $B_j$ and $H_{j+1}$ belong to $\partial \A$, by convexity there are no interior points of $\A$ on the segment $H_{j+1}C$. Noting that $T(C)=(2+a,2-a)$, this completes the proof of $T(\A)\subset \UU_a$, see also Figure~\ref{fig:small-a}. 
\begin{figure}[h]
	\psset{unit=1.7cm}
	\begin{pspicture}(-3,0)(3,4)
	\psline{}(-3,1)(3,1)(0,4)(-3,1)
	\psline{}(-3,3)(3,3)(0,0)(-3,3)
	%\psline{}(0,0)(0,1)
	\psline{}(-0.3,0.3)(2.3,1.7)
	\pscurve[linestyle=dashed]{}(-1,1)(0,0.4615)(1,1)(2,2)(1,3)(-1,3)(-2,2)(-1,1)
	\psline[linestyle=dotted]{}(-0.3,0.3)(0.3,0.3)(2.3,1.7)(2.7,2.7)(2.4,3)(0.3,3.7)(-0.3,3.7)(-2.4,3)(-2.7,2.7)(-2.3,1.7)(-0.3,0.3)
	\psdot(0,0.4615) \uput[90](0,0.4615){$\scriptstyle T(B_j)$}
	\uput{0.12}[135](-1.75,2.35){$\scriptstyle \partial T(\A)$}
	\psdot(1,1) \uput[135](1,1){$\scriptstyle T(H_{j+1})$}
	\psdot(2,2) \uput{0.1}[180](2,2){$\scriptstyle T(H_{j+2})$}
	\psdot(1,3) \uput{0.1}[240](1,3){$\scriptstyle T(H_{j+3})$}
	\psdot(-1,1) \uput[50](-1,1){$\scriptstyle T(H_{j})$}
	\psdot(-2,2) \uput{0.1}[0](-2,2){$\scriptstyle T(H_{j+5})$}
	\psdot(-1,3) \uput{0.1}[-50](-1,3){$\scriptstyle T(H_{j+4})$}
	\psdot(0.3,0.3) \uput[-45](0.3,0.3){$\scriptstyle (a,a)$}
	\psdot(-0.3,0.3) \uput[225](-0.3,0.3){$\scriptstyle (-a,a)$}
	\psdot(2.3,1.7) \uput[0](2.3,1.7){$\scriptstyle T(C)$}
	\psdot(-2.3,1.7) %\uput[190](-2.3,1.7){$Q_9$}
	\psdot(2.7,2.7) \uput[-40](2.7,2.7){$\scriptstyle (3-a,3-a)$}
	\psdot(-2.7,2.7) \uput[220](-2.7,2.7){$\scriptstyle (-(3-a),3-a)$}
	\psdot(2.4,3) \uput[80](2.4,3){$\scriptstyle (3-2a,3)$}
	\psdot(-2.4,3) \uput[100](-2.4,3){$\scriptstyle (-(3-2a),3)$}
	\psdot(2.4,1) \uput[-80](2.4,1){$\scriptstyle (3-2a,1)$}
	\psdot(-2.4,1) \uput[-100](-2.4,1){$\scriptstyle (-(3-2a),1)$}
	\psdot(0.3,3.7) \uput[60](0.3,3.7){$\scriptstyle (a,4-a)$}
	\psdot(-0.3,3.7) \uput[120](-0.3,3.7){$\scriptstyle (-a,4-a)$}	
	\psdot(2.7,1.3) \uput[0](2.7,1.3){$\scriptstyle (3-a,1+a)$}	
	\psdot(-2.7,1.3) \uput[180](-2.7,1.3){$\scriptstyle (-(3-a),1+a)$}	
	\uput[70](1.45,3.25){$\scriptstyle \UU_a$}
	\end{pspicture}
	\caption{Illustration of the inclusion $T(\A)\subset \UU_a$ when $a\in[0,\tfrac12)$}
	\label{fig:small-a}
\end{figure}	
\end{proof}

For our applications, it will be important to understand how to cover $\UU_b$ with a homothetic image of $\LL_a$, for certain values of $a$ and $b$.

\begin{lemma}\label{lem:big-b}
	Let $a\in[0,1]$, $b\in[\frac12,1]$ and $h$ be the homothety with the ratio $2-b$ and the center $(0,2)$. Then $\UU_b \subset h(\LL_a)$. 
\end{lemma}
\begin{proof}
	This is immediate by $\LL_1\subset\LL_a$ and the fact that $h(\LL_1)$ contains all $12$ points defining $\UU_b$. 
\end{proof}

\begin{lemma}\label{lem:aux-comp-coef}
	For any $a,b\in[0,\frac12)$ there exists $c$ such that the homothety $h$ with the ratio $\lambda(a,b)$ and the center $(0,c)$ satisfies $\UU_b\subset h(\LL_a)$, where
	\begin{equation}\label{eqn:lambda-half}
	\lambda({a},{b}):=
	\begin{cases}
	1+\tfrac{(1-{b})(1+2{a})}{2+{a}}, & \text{if } {b} \le \tfrac{3{a}}{2(1+2{a})}, \\
	\tfrac32, & \text{otherwise}.
	\end{cases}
	\end{equation}	
\end{lemma}
\begin{proof}
	First let us fix $c$ satisfying $\max\{\frac{2a}{1+a},\frac{2b}{1+b}\} < c < 2$ and show that $\UU_b\subset h(\LL_a)$ where $h$ is the homothety with the ratio $\lambda$ and the center $(0,c)$ provided
	\begin{equation}\label{eqn:lambda-bound}
	\lambda \ge 1+\max\left\{\frac{2{a}(1-{b})}{c+c{a}-2{a}},\frac{2{b}}{c},\frac{2(1-{b})}{4-c}\right\}.
	\end{equation}
	It is convenient to represent $\LL_a$ in terms of half-planes, namely,
	\begin{equation}\label{eqn:LLa-def}
	\LL_a=\{(x,y):(1+a)y\ge (1-a)|x|+2a,\, y\ge|x|,\, y\le 4-|x|,\, y\le3 \}.
	\end{equation}
	We have $h^{-1}(x,y)=(\tfrac1\lambda x, \tfrac1\lambda(y-c)+c)$, so if $(x,y)$ is a point defining $\UU_b$, by substitution of $h^{-1}(x,y)$ into the inequalities of~\eqref{eqn:LLa-def}, one can compute that
	\begin{equation}\label{eqn:max-lambda}
	\lambda-1\ge \max \left\{
	\frac{|x|-a|x|-y-ay+2a}{c+ac-2a},
	\frac{|x|-y}{c},
	\frac{|x|+y-4}{4-c},
	\frac{y-3}{3-c}
	\right\},
	\end{equation}
	where $\frac{2a}{1+a}<c<2$ was used. Now we substitute the points defining $\UU_b$ into~\eqref{eqn:max-lambda} and simplify the result. For $(x,y)=(b,b)$, since $a,b\in[0,\frac12)$ and $\frac{2a}{1+a}<c<2$, we get
	\begin{equation}\label{eqn:b1}
	\lambda-1\ge \max \left\{
	\frac{2a(1-b)}{c+ac-2a},
	0,
	\frac{-2(2-b)}{4-c},
	\frac{-(3-b)}{3-c}
	\right\} = \frac{2a(1-b)}{c+ac-2a}.
	\end{equation}
	In the same manner, we substitute $(3-b,3-b)$, $(3-2b,3)$, $(b,4-b)$ and $(2+b,2-b)$ into~\eqref{eqn:max-lambda}, and get that $\lambda-1$ is at least $\frac{2(1-b)}{4-c}$, $\frac{2(1-b)}{4-c}$, $\frac{1-b}{3-c}$ and $\max\{\frac{2(b-a)}{c+ac-2a},\frac{2b}{c}\}$, respectively. Therefore, taking~\eqref{eqn:b1} into account and noting that $\frac{2(1-b)}{4-c} > \frac{1-b}{3-c}$ as $c<2$, and that $\frac{2(b-a)}{c+ac-2a}<\frac{2b}{c}$ as $\frac{2b}{1+b}<c$, we obtain the desired~\eqref{eqn:lambda-bound}.
	
It remains to choose $c$ to minimize the right hand side of~\eqref{eqn:lambda-bound}. The functions $f_1(c):=\frac{2{a}(1-{b})}{c+c{a}-2{a}}$ and $f_2(c):=\frac{2{b}}{c}$ are decreasing, while the function $f_3(c):=\frac{2(1-{b})}{4-c}$ is increasing. We can compute and bound the points of intersection: $f_1=f_3$ at $c_1:=\frac{6{a}}{1+2{a}}\in (\frac{2{a}}{1+{a}},2)$ and $f_2=f_3$ at $c_2:=4{b}\in(\frac{2{b}}{1+{b}},2)$. Thus, we can set $\lambda({a},{b}):=1+f_3(\max\{c_1,c_2\})$, which leads to~\eqref{eqn:lambda-half}.
\end{proof}

\section{Bound on diameter -- proof of Theorem~\ref{thm:diam}} \label{sec:diam}

%Now we are ready for the roof of Theorem~\ref{thm:diam}.
	Let $\A,\B\in\C^2$ be arbitrary. By Lemma~\ref{thm:main}, for some $a,b\in[0,1]$ and affine transforms $T_a$ and $T_b$, we have $\LL_a\subset T_a(\A)\subset \UU_a$ and $\LL_b\subset T_b(\B)\subset \UU_b$. Assume $a\ge b$ and fix $\alpha\in(0,\tfrac12)$. When $a\ge\alpha$, i.e., when one of the bodies is ``close'' to an affine-regular hexagon, we will apply a small modification of the construction from~\cite{La}. Otherwise, we will use Lemma~\ref{lem:aux-comp-coef}.
	
	We claim that $d(\A,\B)\le 3-a$. Let $T$ be an affine transform mapping $\LL_1$ to the hexagon $\H$ with the vertices $(0,\frac43)$, $(\pm1,\frac53)$, $(\pm1,\frac73)$, $(0,\frac83)$, and so $\H$ is inscribed into $T(T_b(\B)))$. Note that the lines defining the sides of $\H$ intersect (apart from the vertices of $\H$) at the vertices of $\LL_1$. Therefore, by convexity and Lemma~\ref{thm:main}, \[\H\subset T(T_b(\B))\subset \H_1\subset T_a(\A)\subset \UU_a,\] so to prove our claim it suffices to show that the homothety of $\H$ with the center $(0,2)$ and the ratio $3-a$ contains $\UU_a$, which is straightforward to verify (e.g. the technique of the proof of Lemma~\ref{lem:aux-comp-coef} can be used). 
	
	If ${a},{b}\in[0,\alpha]$, then by Lemma~\ref{thm:main} and Lemma~\ref{lem:aux-comp-coef}, there are homotheties $h_1$ and $h_2$ with the ratios ${\lambda}({a},{b})$ and ${\lambda}({b},{a})$ such that $T_b(\B)\subset h_1(T_a(\A))$ and $T_a(\A) \subset h_2(T_b(\B))$. Then $T_a(\A)\subset h_2(T_b(\B))\subset h_2(h_1(T_a(\A)))$, so ${d}(\A,\B)\le {\lambda}({a},{b}){\lambda}({b},{a})$. 
	
	We can summarize the preceding paragraphs as
	\begin{equation}\label{eqn-to-compute}
	{d}(\A,\B) \le
	\inf_{\alpha\in\left(0,\tfrac12\right)}\max\left\{3-\alpha,\max_{{a},{b} \in [0,\alpha]} {\lambda}({a},{b}){\lambda}({b},{a})\right\}.
	\end{equation}
	
	Now assume ${a},{b}\in[0,\alpha]$ and consider several cases in order to estimate ${\lambda}({a},{b}){\lambda}({b},{a})$.
	If ${b}>\tfrac{3{a}}{2(1+2{a})}$ and ${a}>\tfrac{3{b}}{2(1+2{b})}$, then ${\lambda}({a},{b}){\lambda}({b},{a})=\tfrac94$.
	If ${b}\le\tfrac{3{a}}{2(1+2{a})}$ and ${a}\le\tfrac{3{b}}{2(1+2{b})}$, then ${\lambda}({a},{b}){\lambda}({b},{a})=(1+(1-{b})f({a}))(1+(1-{a})f({b}))$, where $f(t)=\tfrac{1+2t}{2+t}$ which is increasing for $t>-2$. Therefore, $(1-{b})f({a})\le(1-{b})f(\tfrac{3{b}}{2(1+2{b})})=2\tfrac{1+4{b}-5{b}^2}{4+11{b}}$, which, using standard calculus, attains its largest value on $[0,\tfrac12)$ at ${b}=\tfrac{3\sqrt{3}-4}{11}$. Hence,
	$
	1+(1-{b})f({a})\le\frac{289-60\sqrt{3}}{121}\approx 1.52956 < 1.53.
	$
	Arguing similarly for $(1-{a})f({b})$, we obtain
	$
	{\lambda}({a},{b}){\lambda}({b},{a})< 1.53^2 =  2.3409<2.35.
	$
	If ${b}\le\tfrac{3{a}}{2(1+2{a})}$ and ${a}>\tfrac{3{b}}{2(1+2{b})}$, then 
	\[
	{\lambda}({a},{b}){\lambda}({b},{a})=\tfrac32(1+(1-{b})\tfrac{1+2{a}}{2+{a}})\le\tfrac32(1+\tfrac{1+2{a}}{2+{a}})=\tfrac{9(1+{a})}{2(2+{a})}\le\tfrac{9(1+\alpha)}{2(2+\alpha)}.
	\]
	Similarly, if ${b}>\tfrac{3{a}}{2(1+2{a})}$ and ${a}\le\tfrac{3{b}}{2(1+2{b})}$, then ${\lambda}({a},{b}){\lambda}({b},{a})\le\tfrac{9(1+\alpha)}{2(2+\alpha)}$.
	In summary, since $\max\{3-\alpha,2.35,\tfrac94\}\ge 2.5$, \eqref{eqn-to-compute} becomes
	\[
	{d}(\A,\B)\le
	\inf_{\alpha\in\left(0,\tfrac12\right)}\max\left\{3-\alpha,\tfrac{9(1+\alpha)}{2(2+\alpha)}\right\},
	\]
	which is optimal when $3-\alpha=\tfrac{9(1+\alpha)}{2(2+\alpha)}$, i.e. when $\alpha=\tfrac{\sqrt{73}-7}4$. This implies 
	$
	{d}(\A,\B)\le
	\frac{19-\sqrt{73}}4.
	$
%	and the proof is complete.
%\end{proof}

\section{Bound on radius -- proof of Theorem~\ref{thm:rad}} \label{sec:rad}

%\begin{proof}[Proof of Theorem~\ref{thm:rad}.]
Note that $\A=\LL_{\frac14}$, so we need to show for any $\B\in\C^2$ that $d(\A,\B)\le\frac{117}{70}$. Apply Lemma~\ref{thm:main} to $\B$, then for some $b\in[0,1]$ and an affine transform $T$ we have $\LL_b \subset T(\B) \subset \UU_b$. It is rather immediate from Definition~\ref{def:polygons} that if $h_1$ is the homothety with the center $(0,3)$ and the ratio
$\max\{(3-\frac{2b}{1+b})/(3-\frac25),1\},
$ then $h_1(\LL_{\frac14})\subset \LL_b$. If $h_2$ is a homothety with the ratio $\lambda>0$ such that $\UU_b \subset h_2(\LL_{\frac14})$, then by
\[
h_1(\LL_{\frac14})\subset \LL_b \subset T(\B) \subset \UU_b \subset h_2(\LL_{\frac14})
\]
we get
\[
d(\A,\B)\le \frac{\lambda}{\max\{(3-\frac{2b}{1+b})/(3-\frac25),1\}}.
\] 
If $b\in[0,\frac14]$, then by Lemma~\ref{lem:aux-comp-coef} we have $\lambda= \lambda(\frac14,b)\le \lambda(\frac14,0)=\frac53$, so $d(\A,\B)\le\frac53<\frac{117}{70}$. If $b\in[\frac14,\frac12)$, then again by  Lemma~\ref{lem:aux-comp-coef} we have $\lambda= \lambda(\frac14,b)=\frac32$, and as $(3-\frac{2b}{1+b})/(3-\frac25)>\frac{35}{39}$, we obtain $d(\A,\B)<\frac{117}{70}$. Finally, if $b\in[\frac12,1]$, we use Lemma~\ref{lem:big-b} to take $\lambda=2-b$ and then
$
d(\A,\B)\le \frac{13(2-b)(1+b)}{5(3-b)}\le \frac{117}{70}.
$
%\end{proof}

{\bf Acknowledgment.} The authors are grateful to the anonymous referee for several suggestions that improved the paper.

\end{document}